%% file: PRSFS.tex
\documentclass[a4paper]{article}
\input{macros}
\input{environs}
\newcommand{\SFS}{Seifert fibre space}
\newcommand{\pZB}{\ensuremath{ \hat \Z [[\hat B]]}}
\theoremstyle{plain}
\newtheorem*{intro}{Theorem}
\theoremstyle{definition}
\newtheorem*{introdef}{Definition \ref{profrigdef}}
\title{Profinite rigidity for Seifert fibre spaces}
\author{Gareth Wilkes}
\date{December 17, 2015}
\begin{document}
\maketitle
\begin{abstract}
An interesting question is whether two 3-manifolds can be distinguished by computing and comparing their collections of finite covers; more precisely, by the profinite completions of their fundamental groups. In this paper, we solve this question completely for closed orientable Seifert fibre spaces. In particular, all \SFS{}s are distinguished from each other by their profinite completions apart from some previously-known examples due to Hempel. We also characterize when bounded \SFS{} groups have isomorphic profinite completions, given some conditions on the boundary.
\end{abstract}
\section{Introduction}
One possible algorithm to solve the homeomorphism problem for 3-manifolds could run as follows. Given two triangulated 3-manifolds $M_1$ and $M_2$, perform Pachner moves on $M_1$ to try to establish a homeomorphism with $M_2$. In parallel, compute a list of finite-sheeted covers of the two manifolds, along with regularity of the covers and the group of deck transformations. If at some covering degree a difference appears, the two manifolds will be shown to be non-homeomorphic.

The question arises, to what extent will this algorithm work? That is, could the collections of covers of two distinct 3-manifolds have the same structure? This is a manifestation of the wider question of when two groups have the same set of finite quotients. The na\"ive statement in terms of sets of finite quotients is usually replaced with an equivalent formulation concerning the profinite completions of the two groups. The question is then one of `profinite rigidity'. We make the following definition:
\begin{introdef}
An (orientable) 3-manifold is {\em profinitely rigid} if the profinite completion distinguishes its fundamental group from all other fundamental groups of (orientable) 3-manifolds.
\end{introdef} 
 In dimension 2, the analogous property is known to hold by work of Bridson, Conder, and Reid \cite{BCR14}, who showed that the profinite completion distinguishes 2-orbifold groups not just from each other, but from all lattices in connected Lie groups.

For 3-manifolds, only a few examples are known to be profinitely rigid. Bridson and Reid \cite{bridsonreid15} and Boileau and Friedl \cite{boileaufriedl} have proved that the figure-eight knot group is profinitely rigid among 3-manifold groups, along with a handful of other knot groups. By contrast, there are large families known not to be profinitely rigid. Funar \cite{funar13} built on work of Stebe \cite{stebe72} to give infinite families of Sol manifolds with the same finite quotients. Hempel \cite{hempel14} gave Seifert fibred families, with geometry \HR.

These examples notwithstanding, the profinite completion of the fundamental group of a low-dimensional manifolds is known to contain a large amount of information. For instance, Wilton and Zalesskii \cite{WZ14} have shown that the geometry (if any) of a 3-manifold is detected by the profinite completion. In particular, \SFS{}s are distinguished from all other 3-manifolds. Lackenby \cite{lackenby14} has shown that the pro-2 completion of a 3-manifold group determines whether that 3-manifold contains a pair of embedded surfaces which do not disconnect the manifold.

In this paper, we provide the full solution of the profinite rigidity question for closed orientable Seifert-fibred 3-manifolds. In effect, the above-cited examples of Hempel \cite{hempel14} are the only failures of profinite rigidity among these manifolds. The precise statement, when combined with the work in \cite{WZ14}, is:
\begin{intro}
Let $M$ be a (closed orientable) \SFS{}. Then either:
\begin{itemize}
\item $M$ is profinitely rigid; or 
\item $M$ has the geometry \HR{}, is a surface bundle with periodic monodromy $\phi$ and the only 3-manifolds whose fundamental groups have the same finite quotients as $\pi_1 M$ are the surface bundles with monodromy $\phi^k$, for $k$ coprime to the order of $\phi$.
\end{itemize}
\end{intro}

The theorems of \cite{WZ14} are stated for closed manifolds, so we will be a little more circumspect about asserting profinite rigidity among all 3-manifolds. However we may still resolve the rigidity question among \SFS{}s. For the precise statements see Theorems \ref{mainbddthm} and \ref{conversebdd}; in summary
\begin{intro}
Let $M_1$, $M_2$ be \SFS{}s with non-empty boundary. Then the following are equivalent:
\begin{itemize}
\item $\widehat{\pi_1 M_1}\iso \widehat{\pi_1 M_2}$, by an isomorphism inducing an isomorphism of peripheral systems.
\item $M_1$ is a surface bundle with periodic monodromy $\phi$, and $M_2$ is a bundle over the same surface with monodromy $\phi^k$, where $k$ is coprime to the order of $\phi$.
\end{itemize}
\end{intro}

The author would like to thank Marc Lackenby for suggesting this field of study and for many enlightening conversations during the development of this theorem.  The author was supported by the EPSRC and by a Lamb and Flag Scholarship from St John's College, Oxford.

\begin{cnv}
In this document, we will use the following conventions:
\begin{itemize}
\item all manifolds and orbifolds will be assumed compact, connected and without boundary unless otherwise stated; all 3-manifolds will be orientable;
\item abstract groups will be assumed finitely presented and will be denoted with Roman letters $G, H, ...$; they will be assumed to have the discrete topology.
\item profinite groups will be assumed topologically finitely generated and will be denoted with capital Greek letters $\Gamma, \Delta, ...$
\item the symbols \nsgp[f], \nsgp[o], \nsgp[{\mathnormal p}] will denote `normal subgroup of finite index', `open normal subgroup', `normal subgroup of index a power of $p$' respectively; similar symbols will be used for not necessarily normal subgroups.
\item there is a divergence in notation between profinite group theorists, who use $\Z[p]$ to denote the $p$-adic integers, and manifold theorists for whom $\Z[p]$ is usually the cyclic group of order $p$. To avoid any doubt, the cyclic group of order $p$ will be consistently denoted $\Z/p$ or $\Z/p\Z$.
\end{itemize}
\end{cnv}

\input{prelims}
\input{cohom}

\subsection{Chain complexes}
It will be necessary later to work with certain exact sequences of modules over a `group ring' of a profinite group. In this section we will recall and prove some of the necessary tools.
\begin{defn}
Given a profinite abelian group $A$ (usually $\hat \Z$, $\hat \Z_{(p)}$ or a finite abelian group) and a profinite group $\Gamma$, the {\em completed group ring} $A[[\Gamma]]$ is defined as the inverse limit
\[ \varprojlim_{A', N} A/A' [\Gamma/N] \]
of group rings indexed over the finite index open normal subgroups $A', N$ of $A,\Gamma$. It is a compact Hausdorff topological ring. 
\end{defn}
If, for instance, $\Gamma=\hat G$ and $A=\hat\Z$ for $G$ residually finite, $\hat\Z[[\hat G]]$ naturally contains a copy of $\Z{} [G]$ as a dense subring. An abelian group $M$ with a continuous $\Gamma$-action now becomes a (left- or right-) $A[[\Gamma]]$-module in the usual way. 

These modules over $A[[\Gamma]]$, together with continuous module maps, form an abelian category with the same formal properties as the category of $R$-modules for a ring $R$; so the machinery of homological algebra works and we can define profinite group cohomology by starting from an arbitrary resolution of $\hat\Z$ by projective (left) $\hat \Z [[\Gamma]]$-modules and applying the functor ${\rm Hom}_{\hat\Z [[\Gamma]]}(-,M)$ giving the continuous homomorphisms from a module to $M$. If $M$ is a module with trivial $\Gamma$-action, we can factor this through the functor $\hat \Z \otimes_{\hat\Z[[\Gamma]]}-$ which `forgets the $\Gamma$-action' on the chain complex.

We will need to show that, under certain conditions, a free resolution of \Z{} by $\Z{} [G]$-modules yields a free resolution of $\hat \Z$ by $\hat\Z [[\hat G]]$-modules. To this end we use the following propositions and definitions, which are adapted from results in \cite{nakamura94}.

\begin{defn}
A discrete group $G$ is of type FP($n$) if there is a resolution of the trivial module \Z{} by projective $\Z{} [G]$-modules $P_\bullet$, such that $P_i$ is finitely generated for $0\leq i\leq n$.
\end{defn}

\begin{prop}\label{goodgiveslimits}
Let $G$ be a discrete group which is good. Then:
\begin{itemize}
\item $\varinjlim_{K\sbgp[f] G} H^q(K; M) = 0$ for every finite $G$-module $M$ and all $q\geq 1$
\item If $G$ is of type FP($n$), then $\varprojlim_{K\sbgp[f] G} H_q(K;M) = 0$ for every finite $G$-module $M$ and all $1\leq q\leq n$.
\end{itemize}
\begin{proof}
First note we may restrict to the case of trivial modules in the conclusions, as any finite $G$-module $M$ becomes trivial over $K$ for a cofinal subset of $\{K\sbgp[f] G\}$. Thus we may view $M$ interchangeably as a left or right module. The maps ${\rm res}^K_{K'}:H^q(K; M) \to H^q(K';M)$ are given by restriction of cochains. The direct limit in question (categorically a colimit) is zero if all elements of $H^q(K;M)$ are `eventually zero'; that is, for all $x\in H^q(K;M)$ there is some $K'\leq K$ such that $x$ is mapped to zero under the restriction map $H^q(K; M) \to H^q(K';M)$. By goodness of $K$, there is a natural identification $H^q(K;M)\iso H^q(\hat K; M)$ so we may represent $x$ as a continuous cochain $\xi: \hat K^q\to M$ ($q>0$). The preimage of $0$ under $\xi$ is some open subset of $\hat K^q$; products of open subgroups of $\hat K$ form a neighbourhood basis in $\hat K^q$, so we may choose $\Delta \sbgp[o] \hat K$ such that $\xi|_{\Delta^q}=0$; then setting $K'=K\cap \Delta$ (so that $\Delta = \hat K'$) the commuting diagram
\[\begin{tikzcd}
H^q(K;M) \ar{r}{{\rm res}^K_{K'}} & H^q(K';M) \\
H^q(\hat K; M) \ar{u}{\iso} \ar{r}{{\rm res}^{\hat K}_{\hat K'}} & H^q(\hat K'; M) \ar{u}{\iso}
\end{tikzcd}\]
shows that ${\rm res}^K_{K'}(x)=0$; hence that $\varinjlim_{K\sbgp[f] G} H^q(K; M) = 0$.

For the second conclusion, assume $G$ is of type FP($n$). Then $H_q(K;M)$ is finite for all $0\leq q\leq n$, $K\sbgp[f] G$ and $M$ a finite $G$-module, and similarly for the cohomology. By Proposition \ref{limitnonzero} the condition that an inverse limit of finite abelian groups $A_i$ is trivial is equivalent to the existence, for each $i$ in the indexing set, of $j\geq i$ such that $A_j\to A_i$ is the zero map; and similarly for a direct limit of finite groups.

So let $K$ be a finite index subgroup of $G$, and take $K'$ such that the restriction map ${\rm res}^{K}_{K'}$ is zero on each $H^q$. We show that we can dualise this to find that the corestriction map is also zero. Note that a finite-index subgroup of a group of type FP($n$) is also of type FP($n$). Let $P_\bullet$ be a projection resolution of $\Z$ by left $\Z K$-modules, which is finitely generated in dimensions at most $n$. There is a natural isomorphism (see \cite{cartaneilenberg}, Proposition II.5.2)
\[ {\rm Hom}_{\Z K}(P_\bullet, {\rm Hom}_{\Z}(M,\Q/\Z))\iso {\rm Hom}_{\Z}(M\otimes P_\bullet, \Q/\Z)\]
Now take homology; $\Q/\Z$ is an injective abelian group, so ${\rm Hom}(-,\Q/\Z)$ is an exact functor and commutes with homology; hence we get a natural isomorphism
\[ H^q(K; M^\ast)\iso (H_q(K;M))^\ast \]
where $N^\ast$ denotes the dual ${\rm Hom}(N,\Q/\Z)$ of an abelian group. Finite abelian groups are isomorphic to their dual and canonically isomorphic to their double-dual; so we get a natural isomorphism
\[H^q(K; M^\ast)^\ast \iso H_q(K;M)\]
in dimensions $0\leq q\leq n$ where the right hand side is finite. The inclusion $K'\to K$ induces the zero map on the left hand side by assumption, noting that $M$ is isomorphic to $M^\ast$ so the restriction map with $M^\ast$ coefficients also vanishes. Hence the map on the right hand side, the corestriction map, is zero. 
\end{proof}
\end{prop}
 To prove the next proposition, we will need some exactness properties of the functor $\varprojlim$. In general this functor will not be exact and so will not commute with homology. A well-known condition for exactness is the {\em Mittag-Leffler condition}; roughly, it is an `eventual stability' condition. See \cite{weibel} for a full treatment; here we merely state the definition and consequence.
\begin{defn}
An inverse system $(A_i)_{i\in I}$, where $(I,\leq)$ is a totally ordered inverse system (not merely partially ordered) satisfies the {\em Mittag-Leffler condition} if for all $i$ there exists $j\geq i$ such that 
\[{\rm im}(A_k\to A_i) = {\rm im}(A_j\to A_i)\]
for all $k\geq j$. That is, the images of the transition maps into $A_i$ are eventually stable.  
\end{defn} 
If all systems $C_{n, i}$ $(i\in I)$ in an inverse system of chain complexes $C_{\bullet,i}$ satisfy the Mittag-Leffler condition, then we will have 
\[ \varprojlim_i H_n(C_{\bullet,i}) = H_n(\varprojlim_i C_{\bullet, i}) \]
for all $n$. In our case, all the groups $C_{n,i}$ will be finite, so that the Mittag-Leffler condition holds (a decreasing sequence of subsets of a finite set is eventually constant). Our indexing set $I=\{(m,K)\,|\, m\in\N, K\sbgp[f]G\}$ will not be totally ordered; however by passing to the cofinal subset $J=\{(m!, K_n)\}$ where $K_n$ is the intersection of the finitely many subgroups of index at most $n$, we get a totally ordered indexing set without affecting the limit.

\begin{prop}\label{goodgivesexact}
Let $(C_i)_{0\leq i\leq n}\to \Z$ be a partial resolution of \Z{} by free finitely generated $\Z{} [G]$-modules $C_i= \Z{} [G]^{\oplus r_i}$ where $G$ is a good group. Then $(\hat C_i)_{0\leq i\leq n}\to \hat\Z$ is a partial resolution of $\hat \Z$ by free $\hat\Z[[\hat G]]$-modules
\[ \hat C_i = \hat \Z[[\hat G]]^{\oplus r_i} \]
\begin{proof}
For each $m\in\N$ and $K\sbgp[f] G$, set 
\[A_{i,m,K} = (\Z/m) [G/K] \otimes_{\Z{} [G]} C_i = (\Z/m) [G/K]^{\oplus r_i}\]
so that the new chain groups are
\[ \hat C_i = \varprojlim_{m,K} A_{i,m,K} \]
The groups $A_{i,m,K}$ are finite, so the homology of each chain complex $(A_{\bullet,m,K})$ is finite; as described above we may now use the Mittag-Leffler condition to conclude
\[ H_i(\hat C_\bullet) = H_i(\varprojlim_{m,K} A_{i,m,K}) = \varprojlim_{m,K}H_i(A_{i,m,K})\]
 Regarding $(C_\bullet)$ as an exact complex of free finitely generated $\Z{} [K]$-modules and noting that 
\[ A_{i,m,K} = (\Z/m)\otimes_{\Z{} [K]} C_i \]
these homology groups $H_i(A_{i,m,K})$ are precisely $H_i(K;\Z/m)$.
By the goodness of $G$ we can now use Proposition \ref{goodgiveslimits} to conclude
\[ H_i(\hat C_\bullet) = \varprojlim_{m,K}H_i(K;\Z/m) = 0\]
for $n-1\geq i\geq 1$; and for $i=0$
\[ H_0(\hat C_\bullet) = \varprojlim_{m,K}H_0(K;\Z/m) = \varprojlim_{m,K} \Z/m = \hat \Z\]
i.e. $(\hat C_\bullet)$ is a free partial resolution of $\hat \Z$.
\end{proof}
\end{prop}

\input{twoorbs}
\input{sfss}

\section{Theorems}
In this section we prove the following result, which with the work of \cite{WZ14} gives the theorem in the introduction.
\begin{theorem}\label{mainthm}
Let $M_1, M_2$ be (closed orientable) \SFS{}s. Then $\widehat{\pi_1 M_1}\iso\widehat{\pi_1 M_2}$ if and only if one of the following holds:
\begin{itemize}
\item $\pi_1 M_1 \iso \pi_1 M_2$, so that unless they have $\sph{3}$-geometry,  $M_1$ and $M_2$ are homeomorphic; 
\item $M_1$, $M_2$ have the geometry \HR{}, where for some hyperbolic surface $S$ and some periodic automorphism $\phi$ of $S$, the 3-manifolds $M_1$ and $M_2$ are $S$-bundles over the circle with monodromy $\phi$ and $\phi^\kappa$ respectively, where $\kappa$ is coprime to ${\rm order}(\phi)$.
\end{itemize}
\end{theorem}

The non-trivial part of the `if' direction of this theorem was proved by Hempel \cite{hempel14}. Alternatively one can apply the argument of Theorem \ref{conversebdd} to get a new proof.

The solution of the problem will proceed in several stages. Firstly, we will show that, except in the `trivial' geometries, an isomorphism of profinite completions of \SFS{}s will induce an automorphism of the profinite completion of the base orbifold group $\hat B$, which the two \SFS{}s will share; and furthermore that both \SFS{}s will have the same Euler number (up to sign). We will then constrain the automorphism of $\hat B$ and compute the action of such an automorphism on $H^2 \hat B$. Intuitively we will be considering what can happen to the `fundamental class' of the orbifold. We will then be able to conclude the result by considering the cohomology classes giving the \SFS{}s as central extensions of $\hat B$.

The `trivial' geometries mentioned above are $\sph{3}$, $\E^3$, $\sph{2}\times\R$; they are trivial for the profinite rigidity problem in the sense that spherical manifolds have finite fundamental group, and there are only six and two orientable manifolds of the latter two geometries respectively, all distinguished by their first homology. For the rest of the section, a `generic' \SFS{} will mean any \SFS{} not of the above geometries.

We will be using heavily the fact that the subgroup generated by a regular fibre is central; this is only true for orientable base orbifold, so first note that we can reduce to this case as follows. Suppose first that we have a closed \SFS{}. The orbifold group $B$ has a canonical index 2 subgroup corresponding to the orientation cover of the underlying surface of the orbifold. This induces an index 2 cover of the \SFS{}. Note that this cover contains all the information needed to recover the original \SFS{}; in particular, for each exceptional fibre with Seifert invariants $(p,q)$ where $1\leq q< p/2$ the cover has $2$ exceptional fibres with the same invariant $(p,q)$, and has no other exceptional fibres. Because the index 2 subgroup is unique, it will follow that any isomorphism of the profinite completions of the \SFS{} groups will induce an isomorphism for these characteristic covers, to which we may apply the theorem for orientable base orbifold, and then recover the original manifolds. 

When the \SFS{} has boundary, the base orbifold group itself does not distinguish orientable base orbifold from non-orientable, and hence has no obvious characteristic subgroup. However if we assume that the peripheral subgroups of the base orbifold groups are preserved under the isomorphism of profinite completions, we can collapse each of them to obtain a closed orbifold and take the canonical index 2 cover of this, and hence of the original orbifold, to recover the above situation.

\subsection{Preservation of the fibre}
We first prove that the subgroup given by the fibre is still essentially unique for most \SFS{}s. In the statement of the theorem, a `virtually central' subgroup $Z$ of a group $G$ will mean that either $Z$ is central in $G$ or that the ambient group $G$ has an index 2 subgroup containing $Z$ in which $Z$ is central. The fibre subgroup of a \SFS{} subgroup is such a subgroup; it is central when the base orbifold is orientable, or is central in the index 2 subgroup corresponding to the orientation cover of a non-orientable base orbifold.
\begin{theorem}\label{fibrepres}
Let $M,N$ be \SFS{}s and suppose that $\widehat{\pi_1(M)}\iso\widehat{\pi_1(N)}$. Call this common completion $\Gamma$. Then:
\begin{enumerate}
\item $M$ and $N$ have the same geometry;
\item $\Gamma$ has a unique maximal virtually central normal procyclic subgroup unless the geometry of $M$ is \sph{3}, $\sph{2}\times\R$, or $\E^3$; and
\item If the geometry  is \Nil, \HR, or \SLR, then $M$ and $N$ have the same base orbifold and Euler number.
\end{enumerate}
\begin{rmk}
The first conclusion of this theorem was already known by the above-cited theorem of \cite{WZ14}; the proof here, specific to \SFS{}s, is different in some respects, so we include it for completeness.
\end{rmk}
\begin{proof}
As usual, spherical manifolds are distinguished by having finite fundamental groups, hence finite profinite completions. The four model geometries $\E^3$, \Nil, \HR, and \SLR\ are contractible, so the fundamental groups of all such manifolds have cohomological dimension exactly 3. All compact $\sph{2}\times\R$-manifolds are finitely covered by $\mathbb{S}^2\times\mathbb{S}^1$, hence have a finite index subgroup of cohomological dimension 1. All 3-manifold groups are good, so this fact is detected by the profinite completion, hence $\sph{2}\times\R$ is distinguished from the other geometries. Henceforth assume that $M$ has one of the four relevant geometries with contractible universal cover.

Now suppose that $\Gamma$ has two virtually central normal procyclic subgroups, $\overline{<\!h\!>}$ and $\overline{<\!\eta\!>}$, where $h$ is represented by a regular fibre of $M$ and $\overline{<\!\eta\!>}$ is not contained in $\overline{<\!h\!>}$. We will show first that the base orbifold $O$ is Euclidean. Passing to the quotient by $\overline{<\!h\!>}$, the image of $\overline{<\!\eta\!>}$ is a normal procyclic subgroup of $\widehat{\pi_1^\text{orb}(O)}$. By Corollary 5.2 of \cite{BCR14} and Proposition \ref{torselts} above, profinite completions of non-positively curved orbifold groups have no finite normal subgroups, so $\overline{<\!\eta\!>}$ persists as an infinite procyclic subgroup of $\widehat{\pi_1^\text{orb}(O)}$. Hence also the subgroup $\overline{<\!h\!>}$ is still maximal even in the profinite completion i.e. is not contained in some larger normal procyclic subgroup. 

We can now pass to a finite index subgroup of $\Gamma$ whose intersections with $\overline{<\!h\!>}$, $\overline{<\!\eta\!>}$ are central and non-trivial, and then to a further finite index subgroup $\Delta$ so that the corresponding cover of $M$ has base orbifold an orientable surface $\Sigma$ covering $O$. The image of $\overline{<\!\eta\!>}$ now gives a central subgroup of $\widehat{\pi_1 \Sigma}$. But the profinite completion of a surface group has no centre unless the surface is a torus (see \cite{anderson74}, \cite{nakamura94} or \cite{asada01}). Hence $O$ is Euclidean.

The base orbifold $\Sigma$ is now a torus. We know that $\overline{<\!\eta\!>}$ is a central procyclic subgroup of $\widehat{\pi_1 \mathbb{T}^2}\iso \hat{\Z}{}^2$; assume now that it is maximal. Using Theorem 4.3.5 of \cite{ribeszalesskii}, the quotient $\hat{\Z}{}^2/\overline{<\!\eta\! >}$ is $\hat\Z$; hence we can quotient by the closed subgroup $\hat\Z{}^2$ generated by both $h$ and $\eta$ to get an exact sequence 
\[ 1\to \hat\Z{}^2\to \Delta \to\hat{\Z} \to 1\]

We now calculate that
\[ H^1 (\Delta; \Z/n) \iso (\Z/n)^3 \]
for all $n$. As described in section \ref{seccohom} we can calculate this cohomology group using the spectral sequence whose $E_2^{p,q}$ page is given by 
\[E_2^{p,q} = H^p (\hat{\Z}; H^q(\hat{\Z}{}^2; \Z/n))\] 
\[\begin{tikzcd}[row sep = small, column sep = small]
\quad & 3 &0 & 0 & 0\\
& 2 &  \Z/n & \Z/n & 0 \\
q & 1 &  (\Z/n)^2\ar{rrd}& (\Z/n)^2 & 0\\
& 0 &  \Z/n & \Z/n & 0 \\
&& 0 & 1 & 2 & \\
&&&p&& 
\end{tikzcd}\]
Now the only arrow that could alter the $p+q = 1$ diagonal is the arrow shown, which is trivial; so this diagonal is already stable and the first cohomology is $(\Z/n)^3$ as required.

But the finite index subgroup $\Delta\leq\Gamma$ corresponds to a cover $\tilde M\to M$ where the base orbifold of $\tilde M$ is a torus. Then we have 
 \[ \pi_1 \tilde{M} = \big< u_1, v_1, h \,\big|\, [u_1, v_1] = h^{-e}, h\text{ central}\big> \]
where $e$ is the Euler number of $\tilde M$; hence $H_1 \tilde{M} = \Z^{2g}\oplus\Z / e\Z$ and
\[ H^1(\Delta;\Z/n)\iso H^1(\tilde M;\Z/n) \iso (\Z/n)^2 \oplus \Z/{\rm hcf}(e,n)\]
for all $n$. Comparing with the above, we find that $e$ must be zero; by naturality $M$ also has trivial Euler number.

We now deal with the case where $\Gamma$ has a unique maximal virtually central procyclic normal subgroup. Note that in this case, the isomorphism $\widehat{\pi_1(M)}\iso\widehat{\pi_1(N)}$ preserves $\overline{<\!h\!>}$, and hence induces an isomorphism of the profinite completions of the base orbifold; then by Theorem \ref{BCR} and Corollary \ref{BCR+}, $M$ and $N$ have the same base orbifold $O$. 

If we now show that $M$, $N$ have the same Euler number, then we are finished as the geometries are distinguished by base orbifolds and whether the Euler number is non-zero. Again pass to an index $d$ subgroup $\Delta$ of $\Gamma$ with the corresponding cover of $M$ being $\tilde M\to M$; where $M$ has base orbifold a surface. Then, as above, for both $N$ and $M$, the Euler number is given up to sign by the torsion part of $H_1\tilde{M}$, divided by $d$, because the Euler number has the naturality property in Proposition \ref{eulernaturality}. First homology is a profinite invariant, hence $N$ and $M$ have the same Euler number and the proof is complete.
\end{proof}
\end{theorem}
Recall that the Euler number of the \SFS{} was of the form 
 \[ e = -(b+\sum \frac{\beta_i}{\alpha_i}) \]
with $b$ an integer. Thus given the base orbifold (hence the $\alpha_i$) and the Euler number, the only further ambiguity is whether we can change the $\beta_i$ by values $\delta_i$ (with $\delta_i$ not congruent to $0$ modulo $\alpha_i$) such that $\sum \delta_i/\alpha_i$ is an integer. By the Chinese Remainder Theorem, there is no such collection of $\delta_i$ when all the $\alpha_i$ are coprime. Hence we have the following corollary, in which we change notation to follow the usual conventions for cone points.

\begin{clly}\label{noambig}
Let $M$ be a Seifert fibre space whose base orbifold is an orbifold $(\Sigma; p_1,\ldots,p_k)$ where $p_1,\ldots,p_k$ are coprime. Then $\pi_1 M$ is distinguished by its profinite completion from all other 3-manifold groups.
\end{clly} 

The above theorem was stated and proved for closed \SFS{}s. A similar result holds for \SFS{}s with boundary. Much of the above argument holds just as well when the \SFS{} has boundary, except that we must rule out some cases with more than one geometry, and the Euler number is no longer defined (a section of a surface-with-boundary always exists). Furthermore, surfaces are no longer determined by their profinite completion unless we have some information about the boundary.

\begin{theorem}\label{boundedbaseinvariant}
Let $M$, $N$ be Seifert fibre spaces with non-empty boundary. Suppose that $\widehat{\pi_1(M)}\iso\widehat{\pi_1(N)}$. Call this common completion $\Gamma$. Furthermore assume that $M$ and $N$ have the same number of boundary components. Then:
\begin{enumerate}
\item $\Gamma$ has a unique maximal virtually central normal procyclic subgroup unless $M$ (and hence $N$) is a solid torus, $\mathbb{S}^1\times\mathbb{S}^1\times\I$ or the orientable \I-bundle over the Klein bottle; and
\item except in these cases, $M$ and $N$ have the same base orbifold.
\end{enumerate}
\begin{proof}
The only positive Euler characteristic orbifolds with boundary are the disc with possibly one cone point; the \SFS{} is then a fibred solid torus.

The only zero Euler characteristic orbifolds with boundary are the annulus (giving the \SFS{} $\mathbb{S}^1\times\mathbb{S}^1\times\I$), the M\"obius band and disc with two order 2 cone points (both giving the orientable \I-bundle over the Klein bottle). 

These three spaces all have different profinite completions of fundamental groups; one is $\hat \Z$, one is $\hat\Z{}^2$ and the other is non-abelian; and none of the \SFS{}s with hyperbolic base orbifold have virtually abelian fundamental group, so we can safely proceed assuming $M$, $N$ are not any of the three exceptional manifolds above. 

Part 1 of the proposition now follows from the same argument as in Theorem \ref{mainthm}, replacing ``virtually a non-abelian surface group" with ``virtually a non-abelian free group" to get the lack of central subgroups of the base orbifold group. Now the base orbifold groups have isomorphic profinite completions, so by \cite{BCR14}, they are the same group. The ambiguity in surface is now resolved by knowledge that $M$ and $N$ and hence their base orbifolds have the same number of boundary components; and the fact that $\Gamma$ detects whether the unique maximal virtually central normal subgroup $<\!h\!>$ is genuinely central or merely virtually so, hence whether the base orbifold is orientable or not.
\end{proof}
\end{theorem}

\subsection{Central extensions}\label{centralexts}
A {\em central extension} of a group $B$ by a (necessarily abelian) group $A$ consists of a short exact sequence
\[ 1\to A\to G\to B\to 1\]
where the image of $A$ is contained in the centre of $G$. Two such extensions are regarded as equivalent if there is a commutative diagram
\[\begin{tikzcd}
1\ar{r}&\ A \ar{r}\ar{d}{\rm id} & G\ar{r}\ar{d}{\iso} & B \ar{r}\ar{d}{\rm id} &1\\
1\ar{r}&\ A \ar{r}\ & G'\ar{r} & B \ar{r} &1
\end{tikzcd}\]
Equivalence classes of central extensions are classified by elements of $H^2(B;A)$. The proof of this fact proceeds directly via cochains, but for what follows it will also be convenient to have the following interpretation. 

Let $B=\big<x_1,\ldots,x_n\,\big|\,r_1,\ldots,r_m\big>$ be a presentation for $B$, let $F$ be the free group on the $x_i$, and $R$ the normal subgroup generated by the $r_j$. From the Serre spectral sequence for the short exact sequence
\[ 1\to R\to F\to B\to 1\]
we obtain the five-term exact sequence
\[0\to H^1 (B;A) \to H^1 (F;A) \to (H^1 (R;A))^F \to H^2(B;A) \to 0=H^2(F;A)\]
where the third non-zero term denotes those elements of $H^1(R;A)$ invariant under the conjugation action of $F$; in fact this is the group $H^1(R/[R,F];A)$. Given an element of $H^2(B;A)$, lift to some 
\[\xi\in (H^1(R;A))^F = ({\rm Hom}(R,A))^F\]
Then a central extension of $B$ by $A$ is given by the `presentation' (abusing notation slightly):
\[ G = \big< Y, x_1,\ldots, x_n\,\big|\, S,\, Y\subseteq Z(G),\, r_1=\xi(r_1), \ldots, r_m=\xi(r_m)\big>\]
where $A=\big<Y\big| S\big>$. The condition that $A$ does genuinely include into this group is equivalent to the invariance of $\xi$ under the action of $F$. The ambiguity under choice of lift to an element $\xi$ is an element $\psi\in H^1(F;A)$. However this ambiguity corresponds precisely to changing the generating set of $G$ to $Y$ and the elements $x'_i=x_i\cdot \psi(x_i)$. Conversely if two such $G$, $G'$ given by $\xi, \xi'$ are isomorphic {\em by an isomorphism $\Phi$ fixing $B$ and $A$}, then $\xi$ and $\xi'$ differ by $\psi\in H^1(F;A)$ given by $\psi(x_i) = x_i \cdot (\Phi(x_i))^{-1}$.

The question of when two central extensions $G,G'$ of $B$ by $A$ given by $\zeta,\zeta'\in H^2(B;A)$ can be isomorphic allowing arbitrary automorphisms for $B$ and $A$ is more subtle; one needs to prove whether any automorphisms of $B$ and $A$ can carry $\zeta$ to $\zeta'$ by the induced maps on $H^2$. This will be the central issue in the proof of Theorem \ref{mainthm}.

The above theory of central extensions also holds for $\hat B$ profinite, provided that the abelian group $A$ is finite so that the cohomology group $H^2(\hat B,A)$ is reasonably well-behaved. See \cite{ribeszalesskii}. The fundamental groups of generic Seifert fibre spaces (over orientable base) are central extensions
\[ 1\to\Z\to G\to B\to 1\]
classified by an element $\eta_G\in H^2(B;\Z)$, where $B=\ofg[O]$ is the fundamental group of the base orbifold. The profinite completion of a generic \SFS{} group is a central extension of $\hat B$ by the infinite group $\hat\Z$. To avoid the complications raised by the presence of $\hat\Z$, we restrict to a finite coefficient group as follows. Note that since an isomorphism of profinite completions of two \SFS{} groups $G$, $G'$ preserves this central subgroup $\hat\Z$ by Theorem \ref{fibrepres}, and since $\hat\Z$ has a unique index $t$ subgroup, any isomorphism $\hat G\iso \hat{G}'$ induces an isomorphism 
\[\Gamma = \hat G / \overline{<\!h^t\!>}\iso \hat{G}'/\overline{<\!h^t\!>}=\Gamma'\]
where $\Gamma, \Gamma'$ are now central extensions of $\hat B$ by $\Z/t$. Hence they are classified by elements $\zeta,\zeta'$ of $H^2(\hat B;\Z/t)$. But $B$ is a good group, hence $H^2(\hat B;\Z/t)$ is canonically isomorphic to $H^2(B;\Z/t)$; and $\zeta_G,\zeta_{G'}$ are the images of $\eta_G,\eta_{G'}$ under the maps
\[ H^2(B;\Z)\to H^2(B;\Z/t)\iso H^2(\hat B;\Z/t)\]

It remains to show that no automorphisms of $\hat B$ and $\Z/t$ can carry $\zeta_G$ to $\zeta_{G'}$ under the induced maps on $H^2(\hat B;\Z/t)$ for all $t$ unless the manifolds $M_1$, $M_2$ are homeomorphic or are covered by the theorem of Hempel \cite{hempel14}.

Before moving on, let us calculate the cohomology classes $\eta_G$ in terms of the five-term exact sequence; this will be of use later. For a \SFS{} over orientable base with symbol 
 \[(b, \Sigma; (p_1, q_1),\ldots, (p_r, q_r))  \]
the fundamental group has presentation
\begin{multline*}
\big< a_1,\ldots, a_r, u_1, v_1, \ldots, u_g, v_g, h\,\big|\\ h\in Z(\pi_1 M),\, a_i^{p_i} h^{q_i}, \,a_1\ldots a_r[u_1, v_1]\ldots [u_g, v_g] = h^b \big>  
\end{multline*}
Let $1\to R\to F\to B\to 1$ be the corresponding presentation of the base orbifold group. Now $R/[R,F]$ is in fact the free \Z-module on these relations $y_0 = a_1\cdots v_g^{-1}$, $y_i = a_i^{p_i}$; comparing to above general theory we see that the cohomology class $\eta_G$ is the image in $H^2(B;\Z)$ of the map
\[ y_0\mapsto b,\quad y_i\mapsto -q_i\]
in ${\rm Hom}(R/[R,F], A)$. The chain complexes in the following section make rigorous our treatment of $R/[R,F]$ as a free abelian group on these generators. The calculation is similar for the bounded case, except that the $y_0$ term does not appear.

\subsection{Action on cohomology}
We first constrain the possible automorphisms of base orbifold that we need to consider:
\begin{prop}\label{constrainauto}
Let $M_1$, $M_2$ be generic \SFS{}s with $\widehat{\pi_1(M_1)}\iso\widehat{\pi_1(M_2)}$. Let the base orbifold group be 
\[ B = \big<\, a_1, \ldots, a_r, u_1, v_1\ldots, u_g, v_g \,\big|\, a_1^{p_1},\ldots, a_r^{p_r}, a_1\cdots a_r\cdot [u_1,v_1]\cdots [u_g,v_g]\,\big> \]
Then some isomorphism of $\widehat{\pi_1(M_1)}$ with $\widehat{\pi_1(M_2)}$ induces an automorphism of $\hat B$ mapping each $a_i$ to a conjugate of $a_i^{k_i}$, where $k_i$ is coprime to $p_i$.
\begin{proof}
This is a simple corollary of Proposition \ref{torselts}; for the induced automorphism of $\hat B$ from any given isomorphism of the $\widehat{\pi_1(M_i)}$ must induce a bijection on conjugacy classes of maximal torsion elements; hence $a_i$ is sent to a conjugate of $a_{\sigma(i)}^{k_i}$ for some permutation $\sigma$ with $p_{\sigma(i)}=p_i$ and $k_i$ coprime to $p_i$. Permuting the $a_i$ under the permutation $\sigma^{-1}$ is an automorphism of $B$, hence of $\hat B$, so we can force $\sigma$ to be the identity; on the level of the \SFS{}s we are simply relabelling the exceptional fibres and exploiting the invariance of the fundamental group under such relabellings.
\end{proof}
\end{prop}
Note that this proposition works just as well when there is boundary.
\begin{prop}
If $\phi$ is an automorphism of $B$ as in Proposition \ref{constrainauto}, then for any $n$ the action of $\phi^\ast$ on $H^2(\hat B;\Z/n)$ is multiplication by $\kappa$ for some profinite integer $\kappa\in\hat \Z$ such that for all $1\leq i\leq r$, $\kappa$ is congruent to $k_i$ modulo $p_i$.
\begin{proof}
We construct a partial resolution of \Z{} by free $\Z B$-modules, transport this to a partial resolution of $\hat \Z$ by free \pZB-modules, and use this to compute the action on cohomology of the above automorphisms of $B$. Fix the presentation 
\[ B = \big<\, a_1, \ldots, a_r, u_1, v_1\ldots, u_g, v_g \,\big|\, a_1^{p_1},\ldots, a_r^{p_r}, a_1\cdots a_r\cdot [u_1,v_1]\cdots [u_g,v_g]\,\big> \]
of $B$, let $F=F(a_i,u_i,v_i)$ and $R= {\rm ker}(F\to B)$.

Set $C_0 = \Z B$, interpreted as the free \Z-module on the vertices of the Cayley graph of $B$, with $B$-action by left translation on ${\rm Cay}(B)$; the map $\epsilon:\Z B\to\Z$ is the evaluation map. 

Let $C_1 = \Z B\{x_i,\bar{u}_j, \bar{v}_j\}$, the free $\Z B$-module with generators $x_i\,(1\leq i\leq r)$, $\bar{u}_j,\bar{v}_j\,(1\leq j\leq g)$. The generator $x_i$ represents the edge in ${\rm Cay}(B)$ starting at 1 and labelled by $a_i$, and similarly $\bar{u}_j,\bar{v}_j$ represent the edges labelled $u_j,v_j$. Thus $C_1$ is the space of linear combinations of paths in ${\rm Cay}(B)$, with $B$-action given by left-translation.

The boundary map $d_1:C_1\to C_0$ sends each path to the sum of its endpoints, so that for example $x_i \mapsto a_i-1\in\Z B$. Certainly $\epsilon d_1=0$; exactness at $C_0$ now follows by connectedness of the Cayley graph.

Let $C_2=\Z B\{y_0,\ldots, y_r\}$. We can interpret $C_2$ as representing `all the relations of $B$'; that is, all closed loops in the Cayley graph. The generator $y_0$ will represent the relation $a_1\cdots v_g^{-1}$ in the above presentation, and $y_i$ the relation $a_i^{p_i}$; now define $d_2:C_2\to C_1$ by mapping each generator to the loop in the Cayley graph representing it; for instance, 
\begin{align*}
d_2(y_i) &=  x_i + a_1\cdot x_i + a_1^2 \cdot x_i +\cdots + a^{p_i-1}_i \cdot x_i\\
d_2(y_0) &=  x_1 + a_1\cdot x_2 + \cdots + a_1\cdots a_{r-1} \cdot x_r \\
&\quad + a_1\cdots a_r \cdot\bar{u_1} +\cdots - a_1\cdots a_r[u_1,v_1]\cdots [u_g,v_g] \bar{v}_g
\end{align*}  
Any loop in the Cayley graph represents some element of $R$, which can be expressed as a product of conjugates of the relations in the above presentation. Left conjugation of a relation corresponds to left-translating the loop around the Cayley graph; so any such product of conjugates can be realised in the Cayley graph as a $\Z B$-linear combination of the $d_2(y_i)$. Hence $d_1 d_2 = 0$ and the image of $d_2$ is precisely the kernel of $d_1$.

Let us analyse the kernel of $d_2$; let \[s = \sum_i \sum_b n^i_b b \cdot y_i\in {\rm ker}(d_2)\] where $\sum_b n^i_b b\in \Z B$ for each $i$. The coefficient of $x_i$ in $d_2(s)$ is
\[ 0 = \sum_b n^0_b ba_1\cdots a_{i-1} + \sum_b n^i_b b(1+a_i+\cdots a_i^{p_i-1})\]
Multiplying on the right by $(a_i-1)$ kills the second sum; and reparametrising the first sum yields $n^0_{ba_i} = n^0_b$ for all $b\in B$. If $r>1$, the $a_i$ generate an infinite subgroup of $B$; but $\sum n^0_b b$ is a finite linear combination, so $n^0_b=0$ for all $b$. If $r=1$, we can analyse the coefficient of $u_i$ instead as $g>0$ for a non-spherical orbifold; or we can simply note that profinite rigidity in the cases $r=0,1$ was already covered by Corollary \ref{noambig}, so that we need not worry any further about them. We are left to conclude that $\sum_b n^i_b b(1+a_i+\cdots a_i^{p_i-1})=0$, hence $\sum_b n^i_b b$ is some multiple of $(a_i-1)$ and the kernel of $d_2$ is spanned by $(a_i-1)y_i$. 

Now set $C_3 = \Z B\{z_1,\ldots,z_r\}$ and $d_3(z_i) = (a_i-1)\cdot y_i$ to find an exact sequence
\[ C_3\to C_2\to C_1\to C_0\to \Z \]
i.e. a partial resolution of \Z{} by free $\Z B$-modules as desired.

By Proposition \ref{goodgivesexact} we have a partial resolution 
\[ \hat C_3\to \hat C_2\to \hat C_1\to \hat C_0\to \hat \Z \]
where each $\hat C_i$ is the free \pZB-module on the same generators as $C_i$, and the boundary maps are defined by the same formulae on these generators. We can thus use this resolution to compute the first and second (co)-homology on $\hat B$.

Let $\phi:\hat B\to\hat B$ be an automorphism of $\hat B$ as in Proposition \ref{constrainauto}. Construct maps $\phi_{\sharp}:\hat C_i\to \hat C_i$ for $i=1,2$ as follows. Lift $\phi$ to $\tilde\phi: \hat F\to \hat F$ such that 
\[ \tilde\phi(a_i) = (a_i^{k_i})^{g^{-1}_i} \]
for some $g_i\in \hat F$. Write the image of each generator of $\hat F$ under $\tilde\phi$ as a limit of words on these generators; then map the corresponding generator of $\hat C_1$ to the associated limit of paths in the Cayley graph. To define $\phi_{\sharp}$ on $\hat C_2$, note that each relation of $\hat B$ is mapped to an element of $\bar{R}$ under $\tilde\phi$, hence can be written as a (limit of) products of conjugates of relations; now map this to an element of $\hat C_2$ just like before. We have made a choice of expression of an element of $\bar{R}$ in terms of conjugates of relations; the ambiguity is by construction an element of ${\rm ker}(\hat d_2) = {\rm im}(\hat d_3)$, which image will soon vanish. For definiteness, choose
\[ \phi_{\sharp} (y_i) = k_i g_i \cdot y_i \quad (1\leq i\leq r) \]
coming from the obvious expression of $\tilde \phi(a_i^{p_i})$ from above. Because the map on $\bar{R}$ was induced by the map on $\hat F$ used to define $\phi_\sharp :\hat C_1\to \hat C_1$, we get a commuting diagram
\[\begin{tikzcd}
\hat C_3 \ar{r} & \hat C_2 \ar{r} \ar{d}{\phi_\sharp} &\hat C_1 \ar{r}  \ar{d}{\phi_\sharp} &\hat C_0 \\
\hat C_3 \ar{r} & \hat C_2 \ar{r} &\hat C_1 \ar{r} &\hat C_0 
\end{tikzcd}\]
Now apply the functor $\hat \Z \otimes_{\pZB} -$ to the above diagram; i.e., factor out the action of $\hat B$, to get a commuting diagram
\[\begin{tikzcd}
\hat \Z \otimes_{\pZB}\hat C_3 \ar{r}{0} & \hat\Z \{y_0,\ldots, y_r\} \ar{r}{d'_2} \ar{d}{\phi_\sharp} &\hat\Z \{x_i,\bar{u}_j, \bar{v}_j\} \ar{r}{0}  \ar{d}{\phi_\sharp} & \hat \Z \otimes_{\pZB}\hat C_0 \\
\hat \Z \otimes_{\pZB}\hat C_3 \ar{r}{0} & \hat\Z \{y_0,\ldots, y_r\} \ar{r}{d'_2} &\hat\Z \{x_i,\bar{u}_j, \bar{v}_j\}\ar{r}{0} &\hat \Z \otimes_{\pZB}\hat C_0 
\end{tikzcd}\]
with the rows no longer exact, but with the maps marked as zero becoming trivial because the image of each generator of the chain group had a factor $(a_i-1)$. We have some good control over the maps in the above, viz.
\begin{align*}
\phi_\sharp(x_i) &= k_i x_i\\
\phi_\sharp(y_i) &= k_i y_i\\
d'_2(y_0) &= x_1 +\cdots + x_r\\
d'_2(y_i) &= p_i x_i
\end{align*}
If $\phi_\sharp (y_0) = \kappa y_0 +\sum \mu_i y_i$, then tracking this around the diagram we find $$\kappa + p_i \mu_i = k_i$$ for all $i$. 

For $n\in\N$, we now apply ${\rm Hom}_{\hat \Z}(-,\Z/n)$ to the above diagram, to get a commuting diagram 
\[\begin{tikzcd}
{} &\ar{l}{0} (\Z/n) \{y_0^\ast,\ldots, y_r^\ast\}  &(\Z/n) \{x_i^\ast,\bar{u}^\ast_j, \bar{v}^\ast_j\} \ar{l}{d^2} & \ar{l}{0} {} \\
{}  & \ar{l}{0}(\Z/n) \{y^\ast_0,\ldots, y^\ast_r\}  \ar{u}{\phi^\sharp}&\ar{l}{d^2}(\Z/n) \{x_i^\ast,\bar{u}^\ast_j, \bar{v}^\ast_j\}  \ar{u}{\phi^\sharp} &\ar{l}{0} {}
\end{tikzcd}\]
in which the homology of each row gives $H^2(\hat B;\Z/n)$ and $\phi^\sharp$ gives an action on this cohomology group. 

First let us note that this action is genuinely the functorial map $\phi^\ast$ induced by $\phi$. By construction $\hat \Z\otimes_{\pZB}\hat C_2$ is the free $\hat \Z$-module on our relations. In this construction for the discrete group, this would be ${R} / [{R}, F]$. In the profinite world, $\bar{R} / [\bar{R}, \hat F]$ may not be free abelian, as not every closed subgroup of a free profinite group is free; however we do get a canonical surjection
\[ \hat \Z\otimes_{\pZB}\hat C_2 \twoheadrightarrow  \bar{R} / [\bar{R}, \hat F]\]
since our chosen set of relations is a generating set for this latter group. But now the map $\phi_\sharp$ on $\hat \Z\otimes_{\pZB}\hat C_2$ is easily seen to induce the natural map on $\bar{R} / [\bar{R}, \hat F]$ given by $\tilde\phi$; and naturality of the quotient map
\[ H^1 ( \bar{R} / [\bar{R}, \hat F]; \Z/n)\twoheadrightarrow H^2(\hat B;\Z/n)\]
coming from the five-term exact sequence shows that $\phi^\sharp$ will indeed give the correct action on $H^2$.

Finally, we can compute this action on $H^2(\hat B;\Z/n)$. We have from above
\begin{align*}
\phi^\sharp(y^\ast_0) &= \kappa y^\ast_0\\
\phi^\sharp(y^\ast_i) &= \mu_i y^\ast_0 + k_i y^\ast_i\\
d^2(x^\ast_i) &= y^\ast_0 + p_i y^\ast_i\\
d^2(\bar{u}^\ast_i) &= 0 = d^2(\bar{v}^\ast_i)
\end{align*}
so that, given a cochain $\zeta = by^\ast_0 -\sum_i q_i y^\ast_i$, we have
\begin{align*} 
\phi^\ast ([\zeta])=[\phi^\sharp(\zeta)] &= [(\kappa b- \sum q_i\mu_i)y^\ast_0 -  \sum  q_i k_i y^\ast_i]\\
& = [\kappa (b y^\ast_0 - \sum q_i y^\ast_i) - \sum q_i \mu_i(y^\ast_0 + p_i y^\ast_i)]\\
&= \kappa [\zeta]
\end{align*}
\end{proof}
\end{prop}

\begin{proof}[Proof of Theorem \ref{mainthm}]
Recall that we have reduced to the case of orientable base orbifold. As discussed in section \ref{centralexts}, our manifolds $M_1$, $M_2$ are determined by cohomology classes $\eta_1,\eta_2\in H^2(B;\Z)$. If 
\[ M_1 = (b, \Sigma; (p_1, q_1),\ldots, (p_r, q_r))  \]
Then as a cochain in the basis $y^\ast_0,\ldots, y^\ast_r$ of ${\rm Hom}_{\Z B}(C_2,\Z)$ where $C_\bullet$ is the partial resolution defined above, we have (see section \ref{centralexts})
\[ \eta_1 = [by^\ast_0 - \sum_{1\leq i\leq r} q_i y^\ast_i] \]
and similarly for $\eta_2$. From these we get cohomology classes $\zeta_{i,n}\in H^2(\hat B;\Z/n)$. Suppose that $\widehat{\pi_1(M_1)}\iso\widehat{\pi_1(M_2)}$. Then, after possibly reordering the exceptional fibres of $M_2$, we have that the exists $\kappa\in\hat\Z$ such that $\kappa \zeta_{1,n} = \zeta_{2,n}$ for all $n$. It is a consequence of the previous proposition that an automorphism of the base induces such an effect on the cohomology groups; we may also rescale the fibres of the $M_i$ by an automorphism of $\hat \Z$, giving an automorphism of the coefficient ring of $H^2$. But this is simply multiplication of the cohomology class by some element of $\hat{\Z}{}^{\!\times}$, which we merge into $\kappa$.

If the $M_i$ have non-zero Euler number $e>0$ (by reversing the orientation on the fibres we can always force $e>0$ for both manifolds), choose $n= me\prod p_i$ for some integer $m$, and define a group homomorphism $E:H^2(\hat B;\Z/n)\to\Z/n$ by 
\[ E(\sum t_i y^\ast_i) = -t_0\prod p_j + \sum_{i\neq 0} t_i\prod_{j\neq i} p_j \]
so that $E(\kappa \xi)= \kappa E(\xi)$. Since $e=-(b+\sum  q_i/p_i)$, we have $E(\zeta_{1,n}) = e\prod p_j$ modulo $n$; then 
\[E(\kappa\zeta_{1,n}-\zeta_{2,n}) = (\kappa -1)e\prod p_j = 0 \text{ modulo }n\]  
so that $\kappa \zeta_{1,n} = \zeta_{2,n}$ for all $n=me\prod p_j$ can only hold if $\kappa$ is congruent to $1$ modulo $m$ for all $m$, i.e. if $\kappa=1$ and $\zeta_{1,n} = \zeta_{2,n}$ for all $n$, so that $\eta_1=\eta_2$ and $M_1$, $M_2$ are homeomorphic.

If the $M_i$ have Euler number zero, so that they are \HR{} manifolds, choose $n=\prod p_i$ and $k\in\Z$ such that $k$ is congruent to $\kappa$ mod $n$. Then $M_2$ is a \SFS{} with zero Euler characteristic and Seifert invariants $(p_i,kq_i)$; there is only one such, and Hempel showed that these pairs of \HR{} manifolds are precisely those surface bundles in the statement of the theorem.
\end{proof}

Rather easier is the bounded case, given sensible conditions on the boundary. 
\begin{theorem}\label{mainbddthm}
Let $M_1, M_2$ be orientable \SFS{}s with boundary, and assume that there exists an isomorphism $\Phi:\widehat{\pi_1 M_1}\to \widehat{\pi_1 M_2}$ inducing an isomorphism of peripheral systems, in the following sense. The boundary components of $M_1$ determine a conjugacy class of $\Z^2$-subgroups in the fundamental group, which gives a conjugacy class of $\hat\Z{}^2$-subgroups in the profinite completion. The isomorphism $\Phi$ is required to send one such set of conjugacy classes to the other, inducing isomorphisms on the matched copies of $\hat\Z{}^2$. 

Let $M_1$ have Seifert invariants $(p_i, q_i)$. Then for some $k\in \Z$ coprime to all $p_i$, $M_2$ is the \SFS{} with the same base orbifold and Seifert invariants $(p_i, kq_i)$.
\begin{proof}
We can safely focus on hyperbolic base orbifolds, the other three Seifert fibre spaces with boundary being easily distinguished from these and each other by their first homology, hence by the profinite completion. As before, we have already reduced to the case of orientable base orbifold.

Note that two boundary components of the base orbifolds generate distinct free $\hat \Z$ factors of the base orbifold group, and the standard theory of free profinite products (see Theorem 9.1.12 of \cite{ribeszalesskii}) shows that these are not conjugate in the profinite completion; so the number of peripheral conjugacy classes remains the same as the number of boundary components. Then by Theorem \ref{boundedbaseinvariant} both \SFS{}s share the same base orbifold $O$, and there is an induced automorphism of $\hat B = \widehat{\ofg O}$. As before, we can now consider the \SFS{}s as being represented by elements of $H^2(\hat B;\Z/n)$ for arbitrary $n$.

Take a presentation 
\[ B=\big<a_1,\ldots,a_r,b_1,\ldots,b_s,u_1,v_1,\ldots, u_g,v_g\,\big|\, a_1^{p_1},\ldots, a_r^{p_r}\big>\]
for the base orbifold, where the $a_i$ are the cone points and the (conjugacy classes of the) $b_i$ give all but one of the boundary components; the remaining boundary component is 
\[ b_0 = a_1\cdots a_r\cdot b_1 \cdots b_s\cdot [u_1,v_1]\cdots [u_g,v_g] \]
As before, we are at liberty to permute cone points with the same order, and permuting boundary components is also permitted. Thus given Proposition \ref{constrainauto} and the conditions of the theorem we may assume that the  automorphism $\phi$ of $\hat B$ induced by $\Phi$ is of the form 
\[ a_i\mapsto (a_i^{k_i})^{g_i},\, b_j\mapsto (b_j^{l_j})^{h_j}\]
for elements $g_i,h_j$ of $\hat B$, $l_j\in\hat\Z{}^{\!\times}$, and $k_i$ coprime to $p_i$.

Now the induced automorphism of 
\[ H_1(\hat B) = \hat B_{\rm ab} = \bigoplus_{i=1}^r \Z/p_i \oplus \bigoplus_{j=1}^s \hat \Z \oplus \hat\Z{}^{2g}\]
sends the class of $b_0$ to
\[\phi_\ast ([b_0]) = \phi_\ast(\sum [a_i]+\sum_{j\neq 0} [b_j]) = \sum k_i[a_i] + \sum_{j\neq 0} l_j[b_j]\]
and on the other hand to 
\[l_0 [b_0] = \sum l_0[a_i]+\sum_{j\neq 0} l_0[b_j]\]
showing that all the $k_i$ are congruent to $l_0$ modulo $p_i$ and that all the $l_i$ are equal.

Using essentially the same chain complex as in the closed case we can now compute that the action on 
\[H^2(\hat B; \Z/\prod p_i) = \bigoplus_{i=1}^r \Z/p_i\] is multiplication by $l_0$, or equivalently multiplication by some $k\in \Z$ congruent to $l_0$ modulo $\prod p_i$, thus taking the element $(q_1,\ldots, q_r)$ representing $M_1$ to the element representing $M_2$, which we now see to be $(kq_1,\ldots, kq_r)$.
\end{proof}
\end{theorem}
We finally prove the converse to the last theorem. A mild adjustment to this argument, with the appropriate modification of the cohomology group considered, provides another proof of Hempel's theorem on closed \SFS{}s.
\begin{theorem}\label{conversebdd}
Let $M_1, M_2$ be \SFS{}s with non-empty boundary and with the same base orbifold $O$. Suppose $M_1$ has Seifert invariants $(p_i,q_i)$ and $M_2$ has Seifert invariants $(p_i,kq_i)$ where $k$ is some integer coprime to every $p_i$. Then $\widehat{\pi_1 M_1}\iso \widehat{\pi_1 M_2}$.
\begin{proof}
Again it suffices to deal with the case of orientable base orbifold. Let $\Gamma_i =\widehat{\pi_1 M_i}$, let $h_i$ be a generator of the centre of $\pi_1 M_i$, and let 
\[ \Gamma_{i,n} = \Gamma_i / \overline{<\!h^{t}\!>} \]
where $t=n\prod p_i$.

Note that for each $i$ the $\Gamma_{i,n}$ form a natural inverse system with maps $\Gamma_{i,nm}\to \Gamma_{i,n}$. Furthermore, any map from $\Gamma_i$ to a finite group must kill some power of $h$, and hence factors through some $\Gamma_{i,n}$. It follows that
\[ \Gamma_i = \varprojlim_{n} \Gamma_{i,n} \]

Now $k$ maps to an invertible element of $\Z/\prod p_i$; then there is some  invertible element $\kappa$ of $\hat \Z$ congruent to $k$ modulo each $p_i$. One can prove this by noting that by the Chinese Remainder theorem the natural map $(\Z/mn)^{\!\times} \to (\Z/n)^{\!\times}$ is always surjective, hence so is the map $\hat \Z{}^{\!\times}\to (\Z/n)^{\!\times}$.

As discussed in Section \ref{centralexts}, $\Gamma_{i,n}$ is classified by an element
\[\zeta_i \in H^2(\hat B;\Z/t)=\bigoplus_{j=1}^r \Z/p_j\]
where $B$ is the base orbifold group; by assumption $\zeta_2 = k\zeta_1 = \kappa \zeta_1$. Multiplication of the coefficient group $\kappa$ gives an automorphism of the cohomology group taking $\zeta_1$ to $\zeta_2$, hence induces an isomorphism $\Gamma_{1,n}\to \Gamma_{2,n}$. Moreover this isomorphism is compatible with the quotient maps $\Gamma_{i,nm}\to \Gamma_{i,n}$; hence we have an isomorphism
\[ \Gamma_1 = \varprojlim_{n} \Gamma_{1,n}\iso \varprojlim_{n} \Gamma_{2,n} =\Gamma_{2}\]
as required.
\end{proof}
\end{theorem} 
\nocite{thurstonbook}
\nocite{thurstonnotes}

\bibliographystyle{alpha}
\bibliography{SFS}
\end{document}

%% file: macros.tex
\usepackage{amsmath}
\usepackage{amsfonts}
\usepackage{amssymb}
\usepackage{amsthm}
\usepackage{tikz}
\usepackage{tikz-cd}

\newcommand{\sph}[1]{\ensuremath{\mathbb{S}^{#1}}}

\newcommand{\sss}[3]{\ensuremath{\mathbb{S}^{#1}\times\mathbb{S}^{#2}\times\mathbb{S}^{#3}}}

\newcommand{\iso}{\ensuremath{\cong}}
\newcommand{\Z}[1][]{\ensuremath{\mathbb{Z}_{#1}}}
\newcommand{\Q}{\ensuremath{\mathbb{Q}}}
\newcommand{\R}{\ensuremath{\mathbb{R}}}
\newcommand{\I}{\ensuremath{\mathbb{I}}}
\newcommand{\D}{\ensuremath{\mathbb{D}}}
\newcommand{\E}{\ensuremath{\mathbb{E}}}
\newcommand{\N}{\ensuremath{\mathbb{N}}}

\newcommand{\Nil}{\ensuremath{\mathrm{Nil}}}
\newcommand{\HR}{\ensuremath{\mathbb{H}^2\times\R}}
\newcommand{\SLR}{\ensuremath{\widetilde{\mathrm{SL}_2(\R)}}}

\newcommand{\nsgp}[1][]{\ensuremath{\triangleleft_{\rm #1}}}
\newcommand{\sbgp}[1][]{\ensuremath{\leq_{\rm #1}}}
\newcommand{\ofg}[1][]{\ensuremath{\pi_1^{\text{orb}}#1}}
\newcommand{\oec}[1][]{\ensuremath{\chi^{\text{orb}}#1}}

%% file: environs.tex
\newtheorem{theorem}{Theorem}[section]
\newtheorem*{thmquote}{Theorem}
\newtheorem{prop}[theorem]{Proposition}

\newtheorem{lem}[theorem]{Lemma}
\newtheorem{clly}[theorem]{Corollary}
\theoremstyle{definition}
\newtheorem{defn}[theorem]{Definition}
\newtheorem*{cnv}{Conventions}
\theoremstyle{remark}
\newtheorem*{rmk}{Remark}

%% file: prelims.tex
\section{Preliminaries}
\subsection{Inverse limits}
We recall some definitions and facts about inverse limits for use later.
\begin{defn}
Let $I$ be a partially ordered set such that for any $i,j\in I$ there is some $k$ larger than both. An {\em inverse system of groups} $A_i$ is a collection of groups together with maps $\phi_{ji}:A_j\to A_i$ whenever $j\geq i$, such that $\phi_{ii}={\rm id}$ and $\phi_{ji}\circ \phi_{kj}=\phi_{ki}$ whenever $k\geq j\geq i$. The {\em inverse limit} of this system is the group
\[ \varprojlim_{i\in I} A_i = \{(g_i)\in \prod A_i\text{ such that } \phi_{ji}(g_j)=g_i \text{ for all }j\geq i\}\]
\end{defn}
In category-theoretic terms, we have a functor from the poset category $I$ (where there is an arrow from $i$ to $j$ precisely when $i\geq j$) to the category of groups, and we take the limit of this functor; that is, an object $\varprojlim A_i$ equipped with maps $\psi_i$ to each $A_i$ such that $\phi_{ij}\circ \psi_i=\psi_j$ and which is universal with this property.

We can similarly define {\em direct limits} to be colimits of contravariant functors from $I$ to the category of groups.

The category to which we map is of course unimportant for this definition, and the explicit definition above suffices for most concrete categories such as {\sf Sets, Groups, Rings} etc. The following results will be of use later:
\begin{lem}
Let $A_i$ be an inverse system of non-empty finite sets, such that each transition map $\phi_{ji}$ is surjective. Then $\varprojlim A_i$ is non-empty.
\end{lem}
This is in fact a special case of the same result for compact topological spaces, and is essentially equivalent to the finite intersection property. A similar argument proves:
\begin{prop}\label{limitnonzero}
Let $A_i$ be an inverse system of abelian groups. Then $\varprojlim A_i=0$ if and only if for all $i\in I$ there is some $j\geq i$ such that $\phi_{ji}$ is the zero map.
\begin{proof}
The `if' direction is trivial. For the other direction, topologise $\prod A_i$ as the product of discrete sets, and suppose there exists $i\in I$ such that no $\phi_{ji}$ is trivial. For each finite subset $J\subseteq I$, set 
\[ U_J = \{(g_i)\in \prod A_i\text{ such that } \phi_{ji}(g_j)\neq 0 \text{ for all }j\geq i, j\in J\}\]
Each $U_J$ is an open subset of $\prod A_i$. There is some $k$ larger than every element of $J$, and $A_k\to A_i$ is non-trivial; hence, taking some $g_k\in A_k$ not mapping to the zero of $A_i$ and projecting to each $A_j$, we find that $U_J$ is non-empty. A finite intersection of these $U_J$ is again of this form, hence is non-empty. By compactness we conclude that the intersection of all the $U_J$ is non-empty; this intersection meets $\varprojlim A_i$ non-trivially, in those elements of $\varprojlim A_i$ not mapping to zero in $A_i$. Hence $\varprojlim A_i\neq 0$.
\end{proof}
\end{prop}
There is a dual result for direct limits, that a direct limit of finite abelian groups is zero if and only if some map out of each group is trivial.

Since the limit is intuitively determined by the `long term behaviour' of the system, we expect some process analogous to taking a subsequence. The correct notion is as follows.
\begin{defn}
Given a poset $I$ such that for any $i,j\in I$ there is some $k$ larger than both, a  subset $J$ of $I$ is {\em cofinal} if for every $i\in I$ there is $j\in J$ such that $j\geq i$.
\end{defn}
\begin{prop}
Let $A_i$ be an inverse system of groups indexed over $I$, and let $J$ be a cofinal subset. Then
\[ \varprojlim_{i\in I} A_i = \varprojlim_{j\in J} A_j\]
\end{prop}

\subsection{Profinite completions}
The central motivating question of this subject was `how much information is  contained in the finite quotients of a group?'. In this section we discuss how to package that information into a profinite group. Proofs for most of the statements can be found in \cite{ribeszalesskii} or \cite{analyticprop}.
\begin{defn}
Given a suitable class $\cal C$ of finite groups, and a (discrete) group $G$, the \emph{pro-$\cal C$ completion} of $G$ is the inverse limit of the system of groups \[\{G/N\,|\,N\nsgp[f] G,\, G/N\in {\cal C}\}\]
This completion is denoted $\hat{G}_{\cal C}$.

If $\cal C$ is the collection of all finite groups, we simply write $\hat{G}$ for the profinite completion of $G$. If $\cal C$ is the collection of finite $p$-groups, we write $\hat{G}_{(p)}$ and call it the pro-$p$ completion. 

Note that, by the categorical definition of inverse limits, there is a unique natural map $G\to\hat{G}$ induced by the quotient maps $G\to G/N$.

\begin{defn}
A property $\cal P$ of some class of finitely generated groups is said to be a {\em profinite invariant} if, given such finitely generated groups $G_1, G_2$ with isomorphic profinite completions, $G_1$ has $\cal P$ if and only if $G_2$ does.
\end{defn}
\begin{defn}\label{profrigdef}
An (orientable) 3-manifold is {\em profinitely rigid} if the profinite completion distinguishes its fundamental group from all other fundamental groups of (orientable) 3-manifolds.
\end{defn} 
\begin{defn}
A group $G$ is {\em residually finite} if for all $g\in G$ there is a homomorphism $G\to F$ with $F$ finite and the image of $g$ nontrivial.
\end{defn}
Examples of residually finite groups include surface groups (see for example \cite{hempel72}) and all 3-manifold groups (proved by Hempel \cite{hempel87} in the Haken case, which can be extended to all cases using geometrization).
\end{defn}
\begin{prop}
Let $G$ be a (finitely generated) group and $\iota:G\to\hat{G}$ the natural map. Then:
\begin{itemize}
\item the image of $\iota$ is dense;
\item $\iota$ is injective if and only if $G$ is residually finite; and
\item $\iota$ is an isomorphism if and only if $G$ is finite.
\end{itemize}
\end{prop} 
The following proposition gives explicitly the strong links connecting the subgroup structures of groups with isomorphic profinite completions.
\begin{prop}
Let $G_1, G_2$ be finitely generated residually finite groups, and suppose $\phi:\hat{G}_1\to \hat{G}_2$ is an isomorphism of their profinite completions. Then there is an induced bijection $\psi$ between the set of finite index subgroups of $G_1$ and the set of finite index subgroups of $G_2$, such that if $K\sbgp[f] H\sbgp[f] G_1$, then:
\begin{itemize}
\item $[H\,:\,K] = [\psi(H)\,:\,\psi(K)]$;
\item $K\nsgp[]H$ if and only if $\psi(K)\nsgp[]\psi(H)$;
\item if $K\nsgp[]H$, then $H/K\iso\psi(H)/\psi(K)$; and
\item $\hat{H}\iso\widehat{\psi(H)}$.
\end{itemize}
\end{prop}
This follows immediately from the following proposition, which relates the subgroup structure of a group to that of its profinite completion:
\begin{prop}\label{subgroupcorr}
Let $G$ be a finitely generated residually finite group, and $\hat G$ its profinite completion. Identify $G$ with its image under the canonical inclusion $G\hookrightarrow \hat G$. Let $\psi$ be the mapping sending a finite index subgroup $H\sbgp[f] G$ to its closure $\bar{H}$. If $K\sbgp[f] H\sbgp[f] G$ then:
\begin{enumerate}
\item $\psi : \{H\sbgp[f] G\}\to \{U\sbgp[o] \hat G\}$ is a bijection; 
\item $[H\,:\,K] = [\bar{H}\,:\,\bar{K}]$;
\item $K\nsgp[]H$ if and only if $\bar{K}\nsgp[]\bar{H}$;
\item if $K\nsgp[f]H$, then $H/K\iso\bar{H}/\bar{K}$; and
\item $\hat{H}\iso\bar{H}$.
\end{enumerate}
\end{prop}
The open subgroups of $\hat G$ are in fact all its finite index subgroups, by the Nikolov-Segal theorem \cite{nikolovsegal}.

Questions concerning profinite completions are often na\"ively stated in terms of the `set of isomorphism classes of finite quotients' ${\cal C}(G)$. These formulations are in fact equivalent. 
\begin{thmquote}[Dixon, Formanek, Poland, Ribes \cite{dixon82}]
Let $G_1, G_2$ be finitely generated groups. If ${\cal C}(G_1)= {\cal C}(G_2)$, then $\hat G_1\iso \hat G_2$.
\end{thmquote}
\begin{clly}
If $G_1\iso G_2$, then $H_1(G_1;\Z)\iso H_1(G_2;\Z)$.
\end{clly}

It is frequently necessary to have information about subgroups not of finite index. In particular, if $H\leq G$, when is the closure $\bar H$ of $H$ in $\hat G$ isomorphic to the profinite completion of $H$?
\begin{defn}
Let $G$ be a group, $H\leq G$. We say $H$ is {\em separable in $G$}, or that $G$ is {\em $H$-separable} if, given $g\notin H$, there is a finite index subgroup $N$ of $G$ such that $g\notin N$, $H\subseteq N$. 
\end{defn}
This statement is in fact equivalent to `$H = G\cap \bar H$', i.e. $H$ is closed in the topology on $G$ induced from $\hat G$. Furthermore,
\begin{prop}
If $G$ is a group, $H\leq G$, and $G$ is $H_1$-separable for every $H_1\sbgp[f] H$, then $\hat H\iso \bar H$.
\end{prop}
\begin{defn}
A group $G$ is {\em LERF (locally extended residually finite)} if $G$ is $H$-separable for every finitely-generated subgroup $H$ if $G$. 
\end{defn}
Many 3-manifold groups are LERF.
\begin{thmquote}[\cite{scott78}]
The fundamental group of a compact \SFS{} is LERF.
\end{thmquote}
\begin{thmquote}[Agol, Wise, Kahn, Markovic, and others]
The fundamental group of a compact hyperbolic 3-manifold is LERF.
\end{thmquote}

%% file: cohom.tex
\section{Cohomology of profinite groups}\label{seccohom}
Profinite groups have a homology and cohomology theory sharing many features with that for discrete groups; see \cite{ribeszalesskii} or \cite{Serrecohom} for a more full treatment. We provide here only what results we need. One definition, analogous to that for discrete groups, is the following.
Take a profinite group $\Gamma$ and an abelian group $A$ on which $\Gamma$ acts continuously (i.e.\ a $\Gamma$-module). Then define cochain groups $C^n$ and coboundary maps $d:C^n\to C^{n+1}$ by:
\[ C^n(\Gamma,A) = \{\text{continuous functions } f:\Gamma^n\to A\}\]
\begin{align*}
(df)(g_1, \ldots, g_{n+1}) =\, & g_1 \cdot f(g_2,\ldots ,g_{n+1}) \\
& + \sum_{i=1}^{n} (-1)^i f(g_1,\ldots, g_i g_{i+1},\ldots, g_{n+1} )\\
& + (-1)^{n+1} f(g_1,\ldots, g_n)
\end{align*} 
This chain complex gives well-defined cohomology groups $H^n(\Gamma, A)$. The usual functoriality properties still hold, and a notion of cup product is still defined. There are other equivalent definitions. In particular, one can use projective resolutions of $\hat \Z$ and apply functors to compute the (co)homology. The above chain complex may be obtained in this way, and this viewpoint will be exploited later.

\subsection{Goodness and cohomological dimension}
As one might expect from the fact that profinite groups are determined by their finite quotients, this group cohomology often behaves better when $A$ is a finite module. In particular, Serre \cite{Serrecohom} made the following definition:
\begin{defn}
A finitely generated group $G$ is \emph{good} if for all finite $G$-modules $A$, the natural homomorphism 
\[ H^n(\hat{G}; A)\to H^n(G; A) \]
induced by $G\to \hat{G}$ is an isomorphism for all $n$.  
\end{defn}
For the key groups involved in our case, we have:
\begin{thmquote}[Grunewald, Jaikin-Zapirain, Zalesskii \cite{GJ_ZZ08}]
Finitely generated Fuchsian groups are good.
\end{thmquote}
\begin{thmquote}[Grunewald, Jaikin-Zapirain, Zalesskii \cite{GJ_ZZ08}]
Fully residually free groups are good.
\end{thmquote}
Under certain finiteness assumptions which hold in our cases of interest, an extension of a good group by a good group is itself good (see \cite{Serrecohom}); furthermore, finite index subgroups of good groups are good. Hence:
\begin{clly} The fundamental groups of \SFS{}s are good. \end{clly} 
Furthermore virtually fibred 3-manifolds have good fundamental group; with the solution of the Virtual Fibring Conjecture by Agol \cite{agol}, this includes all finite-volume hyperbolic 3-manifolds. Finally, we can piece together these geometric manifold with a theorem of Wilton and Zalesskii \cite{WZ10}, that if all pieces of the JSJ (or prime) decomposition have good fundamental group, so does the whole manifold, to conclude:
\begin{thmquote}[Agol, Wilton-Zalesskii]
Fundamental groups of compact 3-mani\-folds are good.
\end{thmquote}
Note that in the case where the action of $G$ on $A$ is trivial, and $A$ is finite, we have the identifications
\[ H^1(\hat{G}; A) \iso {\rm Hom}(\hat{G}, A) \iso {\rm Hom}(G, A) \iso H^1(G; A) \] 
so goodness is only important when working with higher cohomology groups.

The notion of cohomological dimension of a group also extends well to profinite groups.
\begin{defn}
If $H$ is an abelian group, its \emph{$p$-primary component} $H(p)$ for a prime $p$ is the subgroup of $H$ consisting of all elements whose order is a power of $p$. 
\end{defn}
\begin{defn}
Let $\Gamma$ be a profinite group, and $p$ a prime. The \emph{$p$-cohomological dimension} ${\rm cd}_p(\Gamma)$ of $\Gamma$ is the smallest integer $n$ such that, for all \emph{finite} $\Gamma$-modules $A$, 
\[ H^i(\Gamma; A)(p) = 0 \text{\quad for all } i>n \]
The \emph{cohomological dimension} ${\rm cd}(\Gamma)$ is the supremum of ${\rm cd}_p(\Gamma)$ over all primes $p$.
\end{defn}

Because the cohomological dimension only relies on finite modules in its definition, we also have:
\begin{prop}
If $G$ is a good group, then ${\rm cd}(G) \geq {\rm cd}(\hat{G})$ 
\end{prop}
In particular, we find:
\begin{prop}
If $M$ is a compact aspherical 3-manifold, then ${\rm cd}(\widehat{\pi_1 M}) = 3$.
\begin{proof}
Let $\Gamma = \widehat{\pi_1 M}$. We already have ${\rm cd}(\Gamma) \leq 3$. But $\pi_1 M$ is good and $H^3(M; \Z/2) \iso \Z/2$, so $H^3(\Gamma; \Z/2)(2) \neq 0$, noting that $M$ is aspherical so that the cohomology of $M$ and its fundamental group are the same. Hence ${\rm cd}(\Gamma) \geq {\rm cd}_2(\Gamma) \geq 3$ also.
\end{proof}
\end{prop}
Similarly, as fully residually free groups are good, we find:
\begin{prop}
If $F$ is a free group, and $\Sigma$ is a closed surface with $\chi(\Sigma)<0$, then ${\rm cd}(\hat{F}) = 1$ and ${\rm cd}(\widehat{\pi_1 \Sigma}) = 2$.
\end{prop}

We will also need the following result from \cite{Serrecohom}.
\begin{prop}
Let $p$ be prime, $\Gamma$ a profinite group, and $\Delta$ a closed subgroup of $\Gamma$. Then ${\rm cd}_p (\Delta) \leq {\rm cd}_p (\Gamma)$.
\end{prop}
\begin{clly}\label{torsionfree}
Let $G$ be a residually finite, good group of finite cohomological dimension over \Z. Then $\hat G$ is torsion-free.
\end{clly}

\subsection{Spectral sequence}
Another useful property of the cohomology of profinite groups is that the Serre spectral sequence of group cohomology holds even in the profinite world. That is, give an exact sequence of profinite groups 
\[ 1 \to \Delta \to \Gamma \to \Gamma/\Delta \to 1 \]
and a continuous $\Gamma$-module $A$, then the natural continuous action of $\Gamma$ on $\Delta$ by conjugation descends to an action of $\Gamma/\Delta$ on $H^q(\Delta; A)$. The cohomology of $\Gamma$ is then given by a spectral sequence
\[ E_2^{p,q} = H^p(\Gamma/\Delta; H^q(\Delta; A)) \Rightarrow H^n(\Gamma; A) \]

Any spectral sequence induces an exact sequence in the low-dimensional (co)homology groups; here it is the {\em five term exact sequence}
\begin{align*} 0\to H^1(\Gamma/\Delta; H^0(\Delta;A))& \to H^1(\Gamma;A)\to H^0(\Gamma/\Delta; H^1(\Delta;A))\\
&\to H^2(\Gamma/\Delta; H^0(\Delta;A))\to H^2(\Gamma;A) 
\end{align*}
When $\Gamma$ is a free profinite group and $A$ a trivial module, in particular when a presentation
\[ 1\to R\to F\to G\to 1\]
of an abstract group yields a short exact sequence of profinite groups
\[ 1\to \bar{R}\to \hat F\to\hat G\to 1\]
we get the exact sequence
\[0\to H^1(\hat G;A)\to H^1(\hat F;A)\to H^1(\frac{\bar{R}}{[\bar{R},\hat{F}]}; A) \to H^2(\hat G;A)\to 0\]

%% file: twoorbs.tex
\section{Profinite completions of 2-orbifold groups}
In this section we recall the results of Bridson, Conder and Reid \cite{BCR14} concerning Fuchsian groups (i.e.\! orbifold fundamental groups of hyperbolic 2-orbifolds), and show that they extend to the case of Euclidean 2-orbifolds.

\begin{theorem}[Theorem 1.1 of \cite{BCR14}]\label{BCR}
Let $G_1$ be a finitely-generated Fuchsian group and $G_2$ be a lattice in a connected Lie group. If $\hat{G}_1\iso\hat{G}_2$ then $G_1\iso G_2$.
\end{theorem}
\begin{clly}\label{BCR+}
Let $O_1, O_2$ be closed 2-orbifolds. If $\widehat{\ofg[(O_1)]}\iso\widehat{\ofg[(O_2)]}$ then $\ofg[(O_1)]\iso\ofg[(O_2)]$. If $\oec(O_1)\leq 0$, then $O_1$ and $O_2$ are homeomorphic as orbifolds.
\begin{proof}
Since $\pi_1^\text{orb}(O_1)$ is finite if and only if the orbifold Euler characteristic is positive, we can safely ignore these cases as the profinite completion is then simply the original group. Otherwise, assume $\widehat{\pi_1^{\text{orb}}(O_1)}\iso\widehat{\pi_1^{\text{orb}}(O_2)}$.

The orbifold has a finite cover which is a surface; take such a cover of $O_1$ and the corresponding cover of $O_2$. If necessary pass to a further finite cover so  that both $O_1$ and $O_2$ are covered with degree $d$ by surfaces with isomorphic profinite completions. A surface group is determined by its first homology, which is seen by the profinite completion, so the two surfaces are homeomorphic to the same surface $\Sigma$. Orbifold Euler characteristic is multiplicative under finite covers, so $\chi^{\text{orb}}(O_1) = \chi(\Sigma)/d = \chi^{\text{orb}}(O_2)$. Hence Euclidean and hyperbolic orbifolds are distinguished from each other. 

It only remains, in light of the above theorem of Bridson-Conder-Reid, to distinguish the Euclidean 2-orbifolds from each other. The profinite completion detects first homology; a direct computation shows that this suffices to distinguish all the Euclidean 2-orbifolds except $(\sph{2}; 2, 4, 4)$ and $(\mathbb{P}^2; 2, 2)$. Recall that an isomorphism of profinite completions would induce a correspondence between the index 2 subgroups, with corresponding subgroups having the same profinite completions. But $(\mathbb{P}^2; 2, 2)$ is covered by the Klein bottle with degree 2, and the Klein bottle is distinguished from the other 2-orbifolds by its profinite completion, but does not cover $(\sph{2}; 2, 4, 4)$. So these two Euclidean orbifolds also have distinct profinite completions.
\end{proof}
\end{clly}

\begin{thmquote}[Theorem 5.1 of \cite{BCR14}]
Let $G$ be a finitely generated Fuchsian group. Every finite subgroup of $\hat G$ is conjugate to a subgroup of $G$, and if two maximal finite subgroups of $G$ are conjugate in $\hat G$ then they are already conjugate in $G$.
\end{thmquote}

\begin{prop}\label{torselts}
Let $G$ be the fundamental group of a closed Euclidean 2-orbifold $X$. Every torsion element of $\hat G$ is conjugate to a torsion element of $G$, and if two torsion elements of $G$ are conjugate in $\hat G$ then they are already conjugate in $G$.
\begin{proof}
The second statement is a special case of the fact that a virtually abelian group  is conjugacy separable \cite{stebe72}.

We proceed on a case-by-case basis. If $X$ is a torus or Klein bottle, then $G$ is good and has finite cohomological dimension. Hence $\hat G$ is torsion free by Corollary \ref{torsionfree}. If $X=(\sph{2};2,2,2,2)$ then $G$ is the amalgamated free product of two copies of the infinite dihedral group. The result then holds by the same argument as in Theorem 5.1 of \cite{BCR14}; for a finite subgroup of the fundamental group of a graph of groups must be conjugate into one of the vertex groups, which here are the copies of $\Z/2$. The same result holds profinitely. Similarly if $X=(\mathbb{P}^2;2,2)$ then the fundamental group is an amalgamated free product.

In \cite{BCR14} the triangle orbifolds were dealt with by passing to certain finite covers which decompose as amalgams, and whose fundamental group contains the torsion element of interest. However for Euclidean orbifolds, it may happen that no such covers exist; indeed no orbifold whose fundamental group is an amalgam has any cone points of order greater than 2. We will instead exploit the fact that our triangle groups are virtually abelian. We give in detail the proof for the orbifold $X=(\sph{2};3,3,3)$; the other two triangle orbifolds are similar but involve checking more cases, so it would be uninformative to include the proofs.

Let $G=\big< a,b\,\big|\, a^3,b^3,(ab)^3\big>$. We have a short exact sequence
\[\begin{tikzcd}
1\ar{r} & N\ar{r}\ar{d}{\iso} & G\ar{r}\ar{d}{\rm id} & H_1 G \ar{r}\ar{d}{\iso} & 1\\
1 \ar{r} & \Z^2\ar{r} &G\ar{r} & (\Z/3)^2\ar{r} & 1
\end{tikzcd}\]
The subgroup $N$ is a subgroup of the translation subgroup of $G$. The translation subgroup is generated by the translations $x=a^{-1}b$, $y=ba^{-1}$. The action of conjugation is 
\[x^a=x^b=y^{-1}x^{-1},\quad x^{a^{-1}}=x^{b^{-1}}=y \quad\text{etc}\]
The subgroup $N = \big<\!\big< aba^{-1}b^{-1} \big>\!\big>$ is then generated by the elements $u=y^{-1}x$, $v=x^3$; note that $[a,b]=u^2v^{-1}$. To guide our calculations, note that an element $au^rv^s$ of $G$ acts on the plane by rotation about the centroid of a certain triangle, whose location turns out to be that of the rotation $a$ translated by $u^{r+s}v^{r/3}$ (see Figure \ref{figtorselts}). So in $G$, we have 
\[ au^r v^s = a^{u^{r+s}v^{r/3}} = a^{y^{-r-s}x^s} \]
and we expect similar equations to hold in $\hat G$. 

We have a short exact sequence for $\hat G$ induced from the one above:
\[ 1\to \hat\Z{}^2\to \hat G\to (\Z/3)^2 \to 1\]
and see that any torsion element of $\hat G$ is of the form $a^i b^j u^\rho v^\sigma$ where $i,j =0,1,2$ are not both zero and $\rho,\sigma\in\hat\Z$. For example, take $i=1,j=0$; the other cases are very similar. We now calculate 
\begin{align*}
a^{y^{-\rho-\sigma}x^{\sigma}} & = x^{-\sigma}y^{\rho+\sigma}\cdot a \cdot y^{-\rho-\sigma}x^{\sigma} \\
& = a\cdot (x^{-\sigma}y^{\rho+\sigma})^a y^{-\rho-\sigma}x^{\sigma} \\
& = a\cdot (y^{-1}x^{-1})^{-\sigma}x^{\rho+\sigma}y^{-\rho-\sigma}x^{\sigma}\\
& = a\cdot y^{-\rho}x^\rho x^{3\sigma} = a u^\rho v^\sigma\\
\end{align*}
So that torsion elements of this form are indeed conjugates of elements in $G$. The rest of the proof consists of similar calculations for other cases and can be safely omitted.
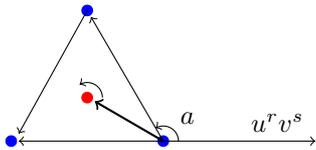
\begin{figure}[htp]
\centering
\begin{tikzpicture}
\draw[fill, blue] (0,0) circle (2pt);
\draw[->] (0.2,0) arc (0:120:0.2);
\draw (0.1,0.3) node[anchor=west]{$a$};
\draw[->] (0,0) -- node[near end,anchor=south] {$u^r v^s$} (2,0) ;
\begin{scope}[rotate=120]
\draw[->] (0,0) -- (1.9,0) ;
\draw[fill, blue] (2,0) circle (2pt);
\end{scope}
\begin{scope}[rotate=-180]
\draw[->] (0,0) -- (1.9,0) ;
\draw[fill, blue] (2,0) circle (2pt);
\end{scope}
\draw[->] (-1,1.732)--(-1.9,0.1);
\draw[fill, red] (-1,0.577) circle (2pt); 
\draw[->] (-0.8,0.577) arc (0:120:0.2);
\draw[thick,->] (0,0) -- (-0.9, 0.527);
\end{tikzpicture}
\caption{The effect of a translation followed by a rotation. The blue points are an orbit of this action; the red point shows the centre of the new rotation. The vector from $a$ to the new centre is $1/3\cdot(u^{-1}v^{-s} + (u^rv^s)^a) $.}
\label{figtorselts}
\end{figure}
\end{proof}
\end{prop}

%% file: sfss.tex
\section{Seifert Fibre Spaces}
We first recall some information about the invariants of a \SFS{} before moving on to profinite matters. For a more comprehensive introduction to Seifert fibre spaces see \cite{brinnotes} and \cite{scott83}.

Recall that a fibred solid torus is formed as a quotient of $\D^2\times[0,1]$ by identifying the two end discs by a rotation by $2\pi q/p$ where $p$, $q$ are coprime integers, called the {\em fibre invariants} of the fibred solid torus. The foliation of $\D^2\times[0,1]$ by lines $\{x\}\times[0,1]$ descends to a foliation of the torus by circles. Such pieces form a local model for a \SFS. Note that the quotient of a fibred solid torus obtained by collapsing each fibre  naturally has an orbifold structure, where the image of the exceptional fibre is a cone point of order $p$. After fixing an orientation for the disc and fibre, the number $q$ becomes well-defined in the range $0<q<p$; if no orientations are chosen, it is well-defined only in the range $0<q\leq p/2$. To give the standard presentation for the fundamental group, it is conventional to define the {\em Seifert invariants} of the exceptional fibre to be $(\alpha, \beta)$ where $\alpha = p$, and $\beta q \equiv 1 \text{ mod }p$.

The orbifold quotients of neighbourhoods of each fibre piece together to form the quotient of the whole manifold $M$ by the foliation; this is the {\em base orbifold} $O$ of the \SFS. This quotient induces a short exact sequence
\[ 1\to <\!h\!>\to \pi_1 M\to\ofg O\to 1\]
where \ofg[O] is the orbifold fundamental group, and $h$ is the element of $\pi_1$ represented by a regular fibre. This subgroup $<\! h\!>$ may be finite or infinite cyclic, and is either central (if $O$ is orientable) or $\pi_1 M$ has an index 2 subgroup which contains $h$ as a central element. 

The final invariant has various different formulations; see \cite{brinnotes}, \cite{scott83}, \cite{NR78}. It is in some sense the `obstruction to a section', and coincides with the Euler number of the fibration when there are no exceptional fibres and the \SFS{} is a {\it bona fide} fibre bundle. In general it is still called the Euler number $e$ of the \SFS{}, and is a rational number. The key properties of the Euler number are the above behaviour when there are no exceptional fibres, and the following naturality property:
\begin{prop}\label{eulernaturality}
If $\tilde M\to M$ is a degree $d$ cover, where the base orbifold cover $\tilde O\to O$ has degree $m$ and a regular fibre of $\tilde M$ covers a regular fibre of $M$ with degree $l$
\[\begin{tikzcd}
1 \arrow{r} &  <\! \tilde h \!> \arrow{r} \arrow{d}{l} & \pi_1 \tilde M \arrow{r} \arrow{d}{d} & \ofg[\tilde O] \arrow{r} \arrow{d}{m} & 1 \\
1 \arrow{r}  &  <\! h \!> \arrow{r} & \pi_1 M \arrow{r} & \ofg[O] \arrow{r} & 1
\end{tikzcd}\]
then 
\[ e(\tilde M) = \frac{m}{l} \cdot e(M) \]
\end{prop}
The Euler number has no well-defined sign {\it a priori}; given a choice of orientation on $M$, $e$ acquires a sign, and reversing the orientation (by flipping the direction along the fibres) changes this sign. This is consistent with the interpretation as the obstruction to a section; when there are no exceptional fibres, circle bundles with orientable total space are classified by elements of $H^2(\Sigma;\Z)$, where the \Z{} coefficients are twisted by the orientation homomorphism for $\Sigma$; this group is \Z{} whether or not $\Sigma$ is orientable.

The vanishing of the Euler number gives important topological information:
\begin{prop}
Let $M$ be a \SFS. The Euler number $e(M)$ vanishes if and only if $M$ is virtually a surface bundle over the circle with periodic monodromy.
\end{prop}
Finally, we can state the classification results of \SFS{}s and characterisations of their fundamental groups from these invariants.
\begin{prop}\label{SFSinvariants}
A \SFS{} is uniquely determined by the symbol
\[ (b, \Sigma; (\alpha_1, \beta_1),\ldots, (\alpha_r, \beta_r)) \]
where
\begin{itemize}
\item $b\in\Z$ and $e = -(b+ \sum \beta_i/\alpha_i)$;
\item $\Sigma$ is the underlying surface of the base orbifold;
\item $(\alpha_i, \beta_i)$ are the Seifert invariants of the exceptional fibres, and $0<\beta_i<\alpha_i$ are coprime.
\end{itemize}
If $\Sigma$ is closed and orientable of genus $g$, $\pi_1 M$ has presentation
\begin{multline*}
\big< a_1,\ldots, a_r, u_1, v_1, \ldots, u_g, v_g, h\,\big|\\ h\in Z(\pi_1 M),\, a_i^{\alpha_i} h^{\beta_i}, \,a_1\ldots a_r[u_1, v_1]\ldots [u_g, v_g] = h^b \big>  
\end{multline*}
If $\Sigma$ is closed and non-orientable of genus $g$, then $\pi_1 M$ has presentation
\begin{multline*}
\big< a_1,\ldots, a_r, v_1, \ldots,v_g, h\,\big|\\ h^{a_i} = h,\, h^{v_i} = h^{-1},\, a_i^{\alpha_i} h^{\beta_i}, \,a_1\ldots a_r v_1^2\ldots v_g^2 = h^b \big>  
\end{multline*}
\end{prop}
When the \SFS{} has boundary, we have similar presentations without the last relation; the base orbifold group is just a free product of (finite or infinite) cyclic groups. In these presentations, $h$ represents the regular fibre; killing $h$ gives a presentation of the orbifold fundamental group of the base. Note also that reversing the orientation of the fibre $h$ and `renormalising' to get the $\beta_i$ back into the correct range, there is an ambiguity in the above symbol for a \SFS, under the transformation
\[ (b,\Sigma; (\alpha_1, \beta_1),\ldots, (\alpha_r, \beta_r)) \to
(-b-r,\Sigma; (\alpha_1, \alpha_1-\beta_1),\ldots, (\alpha_r, \alpha_r-\beta_r))\] 
which flips the sign of $e$. When the orbifold is orientable, this will be the only ambiguity provided there is a unique \SFS{} structure on the manifold.
\begin{prop}
If a closed manifold $M$ has two distinct \SFS{} structures, then it is covered by \sph{3}, $\sph{2}\times\R$, or \sss{1}{1}{1}. 
\end{prop}
\begin{prop}
If $h$ is a regular fibre, then the subgroup $<\!h\!>$ is infinite cyclic unless $M$ is covered by \sph{3}.
\end{prop}
\begin{prop}
A manifold $M$ is Seifert fibred if and only if it has one of the six geometries in Figure \ref{whichgeom}. The geometry is determined by the Euler characteristic of the base orbifold and the Euler number of $M$.
\begin{figure}[ht]
\centering
\begin{tabular}{c|c|c|c|}
& $\oec[]>0$ & $\oec[]=0$ & $\oec[]<0$ \\
\hline
$e=0$ & $\sph{2}\times\R$ & $\E^3$ & \HR \\
\hline
$e\neq 0$ & \sph{3} & \Nil & \SLR \\
\hline
\end{tabular} 
\caption{The geometry of a \SFS{} is determined by the base orbifold and Euler number}
\label{whichgeom}
\end{figure}
\end{prop}

%% file: PRSFS.bbl
\begin{thebibliography}{DDSMS03}

\bibitem[Ago13]{agol}
Ian Agol.
\newblock {The virtual Haken conjecture}.
\newblock {\em {Documenta Mathematica}}, 18:1045--1087, 2013.

\bibitem[And74]{anderson74}
Michael~P. Anderson.
\newblock {Exactness properties of profinite completion functors}.
\newblock {\em {Topology}}, 13(3):229--239, 1974.

\bibitem[Asa01]{asada01}
Mamoru Asada.
\newblock {On centerfree quotients of surface groups}.
\newblock {\em {Communications in Algebra}}, 29(7):2871--2875, 2001.

\bibitem[BCR14]{BCR14}
Martin~R. Bridson, Marston~DE Conder, and Alan~W. Reid.
\newblock {Determining Fuchsian groups by their finite quotients}.
\newblock {\em {arXiv preprint arXiv:1401.3645}}, 2014.

\bibitem[BF15]{boileaufriedl}
Michel Boileau and Stefan Friedl.
\newblock {The profinite completion of $3 $-manifold groups, fiberedness and
  the Thurston norm}.
\newblock {\em {arXiv preprint arXiv:1505.07799}}, 2015.

\bibitem[BR15]{bridsonreid15}
MR~Bridson and Alan Reid.
\newblock {Profinite rigidity, fibering, and the figure-eight knot}.
\newblock {\em {arXiv preprint arXiv:1505.07886}}, 2015.

\bibitem[Bri07]{brinnotes}
Matthew~G. Brin.
\newblock {Seifert Fibered Spaces: Notes for a course given in the Spring of
  1993}.
\newblock {\em {arXiv preprint arXiv:0711.1346}}, 2007.

\bibitem[CE99]{cartaneilenberg}
Henri Cartan and Samuel Eilenberg.
\newblock {\em {Homological algebra}}, volume~19.
\newblock {Princeton University Press}, 1999.

\bibitem[DDSMS03]{analyticprop}
John~D. Dixon, Marcus~PF Du~Sautoy, Avinoam Mann, and Dan Segal.
\newblock {\em {Analytic pro-p groups}}, volume~61.
\newblock {Cambridge University Press}, 2003.

\bibitem[DFPR82]{dixon82}
John~D. Dixon, Edward~W. Formanek, John~C. Poland, and Luis Ribes.
\newblock {Profinite completions and isomorphic finite quotients}.
\newblock {\em {Journal of Pure and Applied Algebra}}, 23(3):227--231, 1982.

\bibitem[Fun13]{funar13}
Louis Funar.
\newblock {Torus bundles not distinguished by TQFT invariants}.
\newblock {\em {Geometry \& Topology}}, 17(4):2289--2344, 2013.

\bibitem[GJZZ08]{GJ_ZZ08}
F.~Grunewald, A.~Jaikin-Zapirain, and P.~A. Zalesskii.
\newblock {Cohomological goodness and the profinite completion of Bianchi
  groups}.
\newblock {\em {Duke Mathematical Journal}}, 144(1):53--72, 2008.

\bibitem[Hem72]{hempel72}
John Hempel.
\newblock {Residual finiteness of surface groups}.
\newblock In {\em {Proc. Am. Math. Soc}}, volume~32, page 323, 1972.

\bibitem[Hem87]{hempel87}
John Hempel.
\newblock {Residual finiteness for 3-manifolds}.
\newblock {\em {Combinatorial group theory and topology}}, 111:379--396, 1987.

\bibitem[Hem14]{hempel14}
John Hempel.
\newblock {Some 3-manifold groups with the same finite quotients}.
\newblock {\em {arXiv preprint arXiv:1409.3509}}, 2014.

\bibitem[Lac14]{lackenby14}
Marc Lackenby.
\newblock {Finding disjoint surfaces in 3-manifolds}.
\newblock {\em {Geometriae Dedicata}}, 170(1):385--401, 2014.

\bibitem[Nak94]{nakamura94}
Hiroaki Nakamura.
\newblock {Galois rigidity of pure sphere braid groups and profinite calculus}.
\newblock 1994.

\bibitem[NR78]{NR78}
Walter~D. Neumann and Frank Raymond.
\newblock {Seifert manifolds, plumbing, {$\mu$}-invariant and orientation
  reversing maps}.
\newblock In {\em {Algebraic and geometric topology}}, pages 163--196.
  {Springer}, 1978.

\bibitem[NS07]{nikolovsegal}
Nikolay Nikolov and Dan Segal.
\newblock {On finitely generated profinite groups, I: strong completeness and
  uniform bounds}.
\newblock {\em {Annals of mathematics}}, pages 171--238, 2007.

\bibitem[RZ00]{ribeszalesskii}
Luis Ribes and Pavel Zalesskii.
\newblock {\em {Profinite groups}}.
\newblock {Springer}, 2000.

\bibitem[Sco78]{scott78}
Peter Scott.
\newblock {Subgroups of surface groups are almost geometric}.
\newblock {\em {Journal of the London Mathematical Society}}, 2(3):555--565,
  1978.

\bibitem[Sco83]{scott83}
Peter Scott.
\newblock {The geometries of 3-manifolds}.
\newblock {\em {Bulletin of the London Mathematical Society}}, 15(5):401--487,
  1983.

\bibitem[Ser13]{Serrecohom}
Jean-Pierre Serre.
\newblock {\em {Galois cohomology}}.
\newblock {Springer Science \& Business Media}, 2013.

\bibitem[Ste72]{stebe72}
P.~F. Stebe.
\newblock {Conjugacy separability of certain Fuchsian groups}.
\newblock {\em {Transactions of the American Mathematical Society}},
  163:173--188, 1972.

\bibitem[Thu78]{thurstonnotes}
William~P. Thurston.
\newblock {\em {The geometry and topology of three-manifolds}}.
\newblock {MIT Press}, 1978.

\bibitem[TL97]{thurstonbook}
William~P. Thurston and Silvio Levy.
\newblock {\em {Three-dimensional geometry and topology}}, volume~1.
\newblock {Princeton university press}, 1997.

\bibitem[Wei95]{weibel}
Charles~A. Weibel.
\newblock {\em {An introduction to homological algebra}}.
\newblock Number~38. {Cambridge university press}, 1995.

\bibitem[WZ10]{WZ10}
Henry Wilton and Pavel Zalesskii.
\newblock {Profinite properties of graph manifolds}.
\newblock {\em {Geometriae Dedicata}}, 147(1):29--45, 2010.

\bibitem[WZ14]{WZ14}
Henry Wilton and Pavel Zalesskii.
\newblock {Distinguishing geometries using finite quotients}.
\newblock {\em {arXiv preprint arXiv:1411.5212}}, 2014.

\end{thebibliography}
